\documentclass[11pt]{amsart}

\usepackage{amssymb, amsrefs}
\usepackage[foot]{amsaddr}
\usepackage{hyperref}
\usepackage{microtype}

\usepackage{tikz-cd}

\usepackage[shortlabels]{enumitem}
\usepackage[margin=1.5in]{geometry}

\newtheorem{theorem}{Theorem}[section]

\newtheorem{lemma}[theorem]{Lemma}
\newtheorem{proposition}[theorem]{Proposition}
\newtheorem{definition}[theorem]{Definition}
\newtheorem{corollary}[theorem]{Corollary}
\theoremstyle{definition}
\newtheorem{remark}[theorem]{Remark}
\newtheorem{question}[theorem]{Question}

\numberwithin{equation}{section}

\newcommand{\N}{\mathbb{N}}
\newcommand{\T}{\mathbb{T}}
\newcommand{\Z}{\mathbb{Z}}

\newcommand{\C}{\mathbb{C}}
\newcommand{\F}{\mathbb{F}}

\newcommand{\U}{\mathcal{U}}

\newcommand\freeprod{\operatornamewithlimits{\ast}}

\address{Department of Mathematics, University of Louisiana at Lafayette, Lafayette, 70503, USA}
\email{lrobert@louisiana.edu}

\title{Selfless C*-algebras}
\author[Leonel Robert]{Leonel Robert}

\begin{document}

\begin{abstract}
The aim of this note is to advertise a class of simple C*-algebras which includes noteworthy examples such as the Jiang-Su C*-algebra, the infinite dimensional UHF C*-algebras, the  reduced group C*-algebra of the free group in infinitely many generators, and the Cuntz algebras. 
\end{abstract}

\maketitle

In \cite{blackadar}, Blackadar considers whether there is a satisfactory extension of the Murray-von Neumann comparison theory for factors to the setting of simple C*-algebras. To this end, he introduces the property of strict comparison of positive elements by 2-quasitraces as a possible point of departure for a comparison theory in the C*-algebraic setting. Blackadar's inspired intuition has been confirmed by developments of the last three decades on the structure theory of simple C*-algebras, notably in the classification program for simple separable nuclear C*-algebras; see \cite{TomsCounterexample, BrownPereraToms, WinterPure, Matui-Sato, GongLinNiu, MannyHands}. Our overarching motivation in this paper has been to further Blackadar's program of investigating C*-algebras with a well-behaved comparison theory.

We introduce a new class of C*-probability spaces which we call \emph{selfless}. These C*-probability spaces are characterized by the existence of a copy of themselves in their ultrapower that is freely independent from the diagonal copy (thus being ``free from themselves''). We develop the foundational aspects of this concept using tools from free probability theory and the model theory of C*-algebras. We have drawn motivation for introducing the selfless class from the work of Dykema and R{\o}rdam on infinite reduced free products \cite{dykema-rordamII}, from their notion of eigenfree C*-probability space \cite{dykema-rordamGAFA}, and from Popa's theorem on the existence of free Haar unitaries in the tracial ultrapower of a II$_1$ factor \cite{popa}.

Given a selfless C*-probability space $(A,\rho)$, if  the state $\rho$ is faithful and non-tracial, then the C*-algebra $A$ is purely infinite, while if $\rho$ is a faithful trace, then $A$ has stable rank one and strict comparison with respect to $\rho$ (see Section \ref{dichotomy_sec}). In either case, $A$ has the strict comparison property. Thus, the selfless property provides a mechanism for obtaining strict comparison. Another by now well-understood source of strict comparison is tensorial absorption of the Jiang-Su C*-algebra \cite{rordam}.

In Section \ref{definition} we define selfless C*-probability spaces and examine multiple characterizations of this property. We show that infinite reduced free products of the kind considered by Dykema and R{\o}rdam (allowing also for possibly non-faithful states) result in selfless C*-probability spaces. In particular, $C_r^*(\mathbb F_\infty)$ is selfless (relative to its unique tracial state).

We obtain several permanence properties of the class of selfless C*-probability spaces (Section \ref{permanence}). Among them, we show that the reduced free product of a selfless $(A,\rho)$ with a $(C,\kappa)$, where $C$ is separable and $\kappa$ has faithful GNS representation, is again selfless. This relies on Skoufranis's ``free exactness'' theorem from \cite{skoufranis}.

In Section \ref{examples} we give several examples of selfless C*-probability spaces. In the non-tracial case, we show that $(A,\rho)$ is selfless whenever $A$ is a Kirchberg algebra and $\rho$ is a pure state. In the tracial case, we show that the Jiang-Su C*-algebra $\mathcal Z$, the  UHF C*-algebras, and the tracial ultrapowers of a separable II$_1$ factor, are selfless. It follows that the reduced free products
\[
(A,\rho)*(\mathcal Z,\tau)
\]
are selfless whenever $(A,\rho)$ is separable and has faithful GNS representation. If $\rho$ is moreover faithful, then these reduced free products are either purely infinite or have stable rank one and strict comparison, paralleling results due to R{\o}rdam for tensor products with $\mathcal Z$ \cite{rordam}.

In Section \ref{ultraembeddings}, we show that the class of separable C*-probability spaces that embed in the ultrapower of a selfless $(A,\rho)$ is closed under reduced free products. This gives a unified way of showing closure under reduced free products for the separable C*-algebras that embed in ultrapowers of $\mathcal O_2$, $\mathcal Q$, $\mathcal Z$, and $R_\omega$.

It has been recently shown in \cite{strongerdixmier} that in a II$_1$ factor every trace zero element can be expressed as a single commutator, and every element can be expressed as the sum of a normal and a nilpotent element. A reasonable analog of a II$_1$ factor in the C*-algebraic setting is a selfless $(A,\rho)$, where $\rho$ is a faithful trace and $A$ has real rank zero (equivalently, $A$ contains projections of arbitrarily small trace). Reinforcing this analogy, we show in Section \ref{commutators} that for such an $(A,\rho)$ the single commutators form a dense set in $\ker \rho$, and that the sums of a normal and a nilpotent element form a dense set in $A$.

A primary motivation for this paper has been the well-known problem of whether the reduced group C*-algebras $C_r^*(\mathbb F_n)$
have strict comparison. This longstanding question has been finally solved in the remarkable work \cite{AGKP}. The authors
in fact show that these C*-algebras, along with a wealth of new examples arising as reduced group C*-algebras,  are selfless.
 
 \textbf{Acknowledgment}: I am grateful to Ilijas Farah for useful feedback on this paper and for prompting me to look into the
 question of axiomatizability of the selfless class. I am grateful to the referees for their helpful suggestions, which 
 improved the exposition of the paper.

\section{Preliminaries on model theory and free probability}
By a C*-probability space we mean a pair $(A,\rho)$, where $A$ is a unital C*-algebra and $\rho$ is a state on $A$. By a  morphism $\theta \colon (A, \rho) \to (B, \tau)$ between C*-probability spaces we understand a unital *-homomorphism between the C*-algebras that preserves the states, i.e.,  such that $\tau \circ \theta = \rho$. We will deal mostly with embeddings of C*-probability spaces, i.e.,  where $\theta$ is assumed to be a C*-algebra embedding.

We will make use of some notions from the model theory of C*-algebras. 
We will briefly summarize some aspects of this theory here, but refer the  reader  to \cite{Cstarmodel} and \cite{OpAlgmodel} for 
further details. 

The language for unital C*-algebras comes equipped with 
\begin{itemize}
	\item a sequence of sorts $(S_n)_{n=1}^\infty$, which in a C*-algebra are interpreted 
	as the closed balls with center the origin and radius $n$, and with the metric given by the norm,
	\item 
	constants 0 and 1, interpreted as the neutral elements of addition and multiplication,
	\item 
	function symbols 
	\begin{align*}
		&+\colon S_m\times S_n\to S_{m+n}, \quad
		 \cdot \colon S_m\times S_n\to S_{mn}, \\
		  &*\colon S_n\to S_n,\quad
\cdot_{\lambda}\colon S_n\to S_{\lceil |\lambda|n\rceil}  \;(\lambda\in\C),
\end{align*}  
for addition, multiplication, involution, and the scalar multiplication by every $\lambda\in \C$ (each equipped with suitable uniform continuity moduli),  and also function symbols  $\iota_{m,n}\colon S_m\to S_n$ for inclusions between different balls, 
	\item 
	no relation symbols (alternatively, a relation symbol can be introduced interpreted as the norm of the C*-algebra).
\end{itemize}	
Formulas are built recursively by continuous connectives and the quantifiers
\(\sup_{\bar x\in S}\) and \(\inf_{\bar x\in S}\) \cite[Definition 2.1.1]{Cstarmodel}. We note that the quantifier free formulas take the form
\[
f\bigl(\|p_1(\bar x)\|,\dots,\|p_k(\bar x)\|\bigr),
\]
where each $p_i$ is a $*$-polynomial in the tuple of variables $\bar x=(x_1,\ldots,x_n)$ (each  $x_j$ is declared to lie in a sort $S_{n_j}$) and $f\colon \mathbb R^k\to\mathbb R$ is continuous. The axioms for the theory of C*-algebras are discussed
in \cite[Example 2.2.1]{Cstarmodel}. We note that they are $\forall$- and $\forall\exists$-sentences.

To build a language for C*-probability spaces, we  enlarge the language of unital C*-algebras with a relation  symbol $\rho$ with modulus of uniform continuity $\Delta_\rho(\epsilon)=\epsilon$. The structures  $(A,\rho)$ of this language  such  that $A$ is a unital C*-algebra and $\rho$ is a state on $A$ form an axiomatizable class. Indeed, the list of axioms for these structures  consists of the axioms for unital C*-algebras together with  $\forall$-sentences enforcing that $\rho$ is linear, positive, and preserves the unit.
Quantifier-free formulas in this expanded language are built by applying continuous functions to finitely many 
expressions of the forms $\|p(\bar x)\|$ and  $\rho(p(\bar x))$.


\begin{definition}
A unital embedding of C*-probability spaces $\theta\colon (A,\rho) \to  (B,\tau)$ is called existential if for any quantifier-free formula $\Phi(\bar x,\bar y)$ in the language of C*-probability spaces and  tuple $\bar a$ in the unit ball of  $A$, we have
\begin{equation}\label{pos-exist}
\inf_{ \bar y} \Phi(\bar a,\bar y)^{(A,\rho)}=\inf_{\bar y} \Phi(\theta(\bar a),\bar y)^{(B,\tau)}, 
\end{equation}
where the tuple $\bar y$ ranges through the unit ball of $A$ on the left-hand side, and the unit ball of $B$ on the right-hand side.
\end{definition}

Given an  ultrafilter $\U$ (on a set $I$),  the ultrapower C*-probability space  $(A, \rho)^\U$   is $(A^\U, \rho^\U)$, where
\[
\rho^\U([(x_i)_{i\in I}])=\lim_\U \rho(x_i).
\]
A less model-theoretically inclined reader may prefer the following characterization of existential embedding, which avoids reference to formulas in the language of C*-probability spaces: $\theta\colon (A,\rho)\to (B,\tau)$ is existential  if there exists an ultrafilter $\U$ and an \emph{embedding}  $\sigma\colon (B,\tau)\to (A^{\U},\rho^{\U})$ into the ultrapower of $(A,\rho)$, such that $\sigma\theta$ agrees with the diagonal embedding of $(A,\rho)$ in $(A^\U,\rho^{\U})$:
\begin{equation}\label{sigmatheta}
\begin{tikzcd}
(A, \rho) \arrow[rr, "\iota"] \arrow[rd, "\theta"] & & (A^{\mathcal{U}}, \rho^{\mathcal{U}}) \\
& (B, \tau) \arrow[ru, "\sigma"] &
\end{tikzcd}
\end{equation}
See \cite[Section 2]{goldbring-sinclair} for a discussion of existential embeddings,  and  their relation to approximate splitting
through the ultrapower; see also \cite[Theorem 4.19]{barlak-szabo} for the related notion of sequentially split embedding.

We will make use of some elementary facts about existential embeddings that we state in the form of three lemmas, for ease of reference.

\begin{lemma}\label{lemexistential}
Let $\theta\colon (A,\rho)\hookrightarrow (B,\tau)$ and $\sigma\colon (B,\tau)\hookrightarrow (C,\kappa)$ be C*-probability space embeddings.
\begin{enumerate}[(i)]
\item
If $\sigma\theta$ is existential, then $\theta$ is existential.
\item
If both $\theta$ and $\sigma$ are existential, then so is $\sigma\theta$.
\end{enumerate}
\end{lemma}
\begin{proof}The proof is an easy exercise.\end{proof}

Given C*-probability spaces $(A_i,\rho_i)_{i\in I}$, and an ultrafilter $\U$ on $I$, their ultraproduct $\prod_\U (A_i,\rho_i)$ is 
the C*-probability space $(\prod_\U A_i,\rho_\U)$, where $\rho_\U$ is the limit of the states $(\rho_i)_i$ along $\U$.  

\begin{lemma}\label{ultraexist}
	If the embeddings $\theta_i\colon (A_i,\rho_i)\hookrightarrow (B_i,\tau_i)$ are existential for all $i\in I$, and
	$\U$ is an ultrafilter on $I$, then the induced embedding $\theta_U\colon \prod_\U (A_i,\rho_i)\hookrightarrow \prod_\U (B_i,\tau_i)$
	is existential
\end{lemma}	

\begin{proof}
Let $\Phi(\bar x,\bar y)$ be a quantifier-free formula and let $\bar a\in \prod_\U A_i$. Then
\begin{align*}
\inf_{\bar y} \Phi(\theta_\U(\bar a), \bar y)^{(\prod_\U (B_i,\tau_i))} &=\lim_\U \inf_{\bar y} \Phi(\theta_i(\bar a_i),\bar y)^{(B_i,\tau_i)}\\
&= \lim_\U\inf_{\bar y} \Phi(\bar a_i,\bar y)^{(A_i,\rho_i)}\\
& = \inf_{\bar y} \Phi(\bar a, \bar y)^{(\prod_\U (A_i,\rho_i))},
\end{align*}
where we have used \L{}o\'s's theorem \cite[Theorem 2.3.1]{Cstarmodel} at the first and last equality, and the fact that the embeddings are existential at the intermediate equality.	
\end{proof}	

For the purpose of simplifying the statement of the following lemma, let us introduce a variation on the notion of being existential: Given $(A,\rho)$,  a C*-subalgebra $C\subseteq A$, and an embedding $\theta\colon (C,\rho|_{C})\hookrightarrow (B,\tau)$, let us  say that $\theta$ is relatively existential in $(A,\rho)$ if \eqref{pos-exist}  holds for all tuples $\bar a$ in $C$. This is equivalent to asking that  for some ultrafilter $\U$ there exists an embedding $\sigma\colon (B,\tau)\to (A^{\U},\rho^{\U})$ such 
that $\sigma\theta$ coincides with the diagonal embedding of $(C,\rho|_{C})$ in $(A^{\U},\rho^{\U})$.

\begin{lemma}\label{unusedsymbol}
Let  $\theta\colon (A, \rho)\to (B,\tau)$ be an embedding of C*-probability spaces. Suppose that  
$A=\overline{\bigcup_{i\in I} A_i}$ and $B=\overline{\bigcup_{i\in I} B_i}$, for upward directed families of unital C*-subalgebras
$(A_i)_{i\in I}$ and $(B_i)_{i\in I}$ such that $\theta(A_i)\subseteq B_i$ and the embeddings 
\[
\theta|_{A_i}\colon (A_i,\rho|_{A_i})\hookrightarrow (B_i,\tau|_{B_i})
\] 
are relatively existential in $(A,\rho)$ for all $i$. Then $\theta$ is existential.
 \end{lemma}
\begin{proof}
Cf. \cite[Proposition 2.7]{barlak-szabo}.

Fix a quantifier-free formula $\Phi(\bar x, \bar y)$. To verify \eqref{pos-exist}, it suffices to let the entries of the tuple $\bar a$ range in the unit ball of the dense subset $\bigcup_{i\in I}  A_i$. Assume thus that
 $\bar a$ is a tuple of elements in the unit ball of $A_i$ for some $i$. Then  
\begin{align*}
\inf_{\bar y} \Phi(\theta(\bar a),\bar y)^{(B,\tau)} &=
\inf_{j\geq i} \inf_{\bar y} \Phi(\theta(\bar a),\bar y)^{(B_j,\tau)}\\  
&\geq \inf_{j\geq i} \inf_{\bar y} \Phi(\bar a,\bar y)^{(A, \rho)} = \inf_{\bar y}\Phi(\bar a,\bar y)^{(A,\rho)}.\qedhere
\end{align*}
\end{proof}

We will make use of some notions from free probability theory. We refer the  reader  to \cite{dykema-nica-voiculescu} for an introduction to  this theory.
Given C*-probability spaces $(A_i,\rho_i)_{i\in I}$, where each $\rho_i$ induces a  faithful GNS representation,  
their reduced free product 
\[
(A,\rho)=\freeprod_{i\in I}(A_i,\rho_i)
\] 
is a C*-probability space equipped with embeddings  $\theta_i\colon (A_i,\rho_i)\to (A,\rho) $, for $i\in I$, such that 
the C*-subalgebras  $\{\theta_i(A_i):i\in I\}$ are freely independent relative to the free product state $\rho$  \cite{avitzour,voiculescu}. All the free products that we shall consider are reduced free products.  Sometimes, by a slight abuse of notation, we write 
$(\freeprod_{i\in I}A_{i},\freeprod_{i\in I}\rho_{i})$,  or simply $\freeprod_{i\in I}A_{i}$ if the states $\rho_i$ have been fixed or can be inferred from the context.

\begin{remark}
We shall always assume, of our C*-probability spaces, that the states induce faithful GNS representations. 
The only exception to this rule  will be  C*-probability spaces obtained as an ultrapower 
or an ultraproduct.
\end{remark}

We will frequently, and tacitly, use the following theorem of Blanchard and Dykema \cite[Theorem 3.1]{blanchard-dykema},
ensuring that we can take reduced free products of embeddings of C*-probability spaces:

\begin{theorem}\label{thm-blanchard-dykema}
Let $\theta_i\colon (A_i, \rho_i) \hookrightarrow (B_i, \tau_i)$, for $i\in I$, be C*-probability space embeddings, where the states $\rho_i$ and $\tau_i$ induce faithful GNS representations for all $i$. Set  $(A, \rho) = \freeprod_i (A_i, \rho_i)$ and $(B, \tau) = \freeprod_i (B_i, \tau_i)$. Then, there exists an embedding $\pi \colon (A, \rho) \to (B, \tau)$ such that the following diagrams commute for all $i$:
\[
\begin{tikzcd}
B_i\ar[r, hook] & B\\
A_i\ar[u, hook, "\theta_i"]\ar[r,hook] & A\ar[u, hook, "\pi"] 
\end{tikzcd}
\]
\end{theorem}
We denote the embedding $\pi$   by $\freeprod_i \theta_i$.


We will make use of a theorem first obtained by  Skoufranis (\cite[Theorem 3.1]{skoufranis}), and in a more general form by  Pisier (\cite[Corollary 4.3]{pisier}), amounting to the fact that in the context of reduced free products, the  analogue of the property of exactness (from the theory of  tensor products of C*-algebras) always holds, i.e., free exactness comes for free. 

Let us  first recall the notion of strong convergence of indexed families of elements. 
 
 Let $\{a_j^{(n)} : j \in J\}$ and $\{a_j : j \in J\}$ be countably indexed families of elements in C*-probability spaces $(A_n, \rho_n)$ and $(A, \rho)$, respectively, where $n=1,2,\ldots$. Let $\U$ be a nonprincipal ultrafilter on $\N$. 
Let us say that $\{a_j^{(n)} : j \in J\}$ converges strongly to $\{a_j : j \in J\}$ along $\U$, denoted
\begin{equation}\label{strongconv}
\begin{tikzcd}
\{a_j^{(n)} : j \in J\}\ar{r}{s}[swap]{\U}& \{a_j : j \in J\}, 	
\end{tikzcd}		
\end{equation}
if for any non-commutative polynomial $p$ in the variables $\{x_j, x_j^* : j \in J\}$, we have:
\begin{enumerate}
    \item $\lim_\U \rho^{(n)}(p(a_j^{(n)}))  = \rho(p(a_j))$,
    \item $\lim_\U \|p(a_j^{(n)})\|  =  \|p(a_j)\|$.
\end{enumerate}
Note that the first condition, and faithfulness of the GNS representations, implies that
\[
\lim_\U \|p(a_j^{(n)})\| \geq \|p(a_j)\|.
\]
Thus, we can replace the second condition in the definition of strong convergence with 
\begin{enumerate}
\item[(2')]
$\lim_\U \|p(a_j^{(n)})\| \leq \|p(a_j)\|$.
\end{enumerate}

Assume that $A=C^*(a_j:j\in J)$. Then from the strong convergence \eqref{strongconv} along $\U$ we readily obtain an embedding of
 C*-probability spaces
\[
\pi\colon (A, \rho) \to \prod_\U(A_n,\rho_n):=\Big(\prod_\U A_n, \rho_\U\Big),
\]
defined by
\[
a_j \mapsto [(a_j^{(n)})_n]_\U,
\]
where $\rho_{\U}$ denotes the limit of $(\rho^{(n)})_n$ along $\U$.

We can now state the above mentioned  theorem by Skoufranis and Pisier. 

\begin{theorem} Let
$(A,\phi)$, $(A_n,\phi_n)$, $(B,\psi)$, and $(B_n,\psi_n)$, for $n=1,2,\ldots$,   
be separable C*-probability spaces whose states induce faithful GNS representations.
Let $\{a_i:i\in I\}$, $\{a_i^{(n)}:i\in I\}$, $\{b_j:i\in J\}$ $\{b_j^{(n)}:j\in J\}$ be countably indexed families of generators 
for the C*-algebras of each of these C*-probability spaces. Suppose that
\[
\begin{tikzcd}
	\{a_i^{(n)} : i \in I\}\ar{r}{s}[swap]{\U}& \{a_i : i \in I\} 		
\end{tikzcd} \hbox{ and }
\begin{tikzcd}
	\{b_j^{(n)} : j \in J\}\ar{r}{s}[swap]{\U}& \{b_j : j \in J\}. 	
\end{tikzcd}	
\]
Then
\[
\begin{tikzcd}
\{a_i^{(n)}, b_j^{(n)} : i\in I, \, j \in J\}\ar{r}{s}[swap]{\U}& \{a_i, b_j : i\in I, \, j \in J\}. 	
\end{tikzcd}	
\]
In this convergence the indexed families are taken in $(A_n,\phi_n)*(B_n,\psi_n)$
on the left-hand side, and in $(A,\phi)*(B,\psi)$ on the right-hand side.
\end{theorem}
	 
\begin{proof}
	This is \cite[Corollary 4.3]{pisier}, except that there the strong convergence is taken in the ordinary sense (i.e., relative to the Fr\'echet filter), rather than with respect to an ultrafilter. The proof of \cite[Corollary 4.3]{pisier}, however, appeals directly to \cite[Theorem 4.1]{pisier}, which is in turn proven first for strong convergence along ultrafilters, and then for ordinary convergence. From
	the ultrafilter version of \cite[Theorem 4.1]{pisier} we readily obtain the ultrafilter version of \cite[Corollary 4.3]{pisier}, i.e., the theorem that we have recalled here.
\end{proof}


\begin{theorem}\label{freeexactness}
	Let $(B_1,\tau_1)$ and $(B_2, \tau_2)$  be separable C*-probability spaces with faithful GNS representations. Suppose that, for an ultrafilter $\mathcal{U}$ on $\mathbb{N}$ and $i=1,2$, we have embeddings 
	\[
	(B_i, \tau_i) \hookrightarrow \prod_\U  (A_i^{(n)}, \rho_i^{(n)}),
	\]
where $(A_i^{(n)},\rho_i^{(n)})$, for $n=1,2,\ldots$, are  C*-probability spaces with faithful GNS representations.
	Then, there exists an embedding
	\[
	\pi\colon (B_1, \tau_1) \ast (B_2, \tau_2) \hookrightarrow \prod_\U  (A_1^{(n)}*A_2^{(n)}, \rho_1^{(n)}*\rho_2^{(n)})
	\]
	such that $\pi|_{B_i}$ agrees with the embedding of $B_i$ into $\prod_\U A_i^{(n)}$ for $i = 1,2$.
	In other words, the following diagram commutes for $i = 1,2$:
	\[
	\begin{tikzcd}
		\prod_\U  A_i^{(n)} \ar[r, hook] & \prod_\U  (A_1^{(n)}*A_2^{(n)}) \\
		B_i \ar[u, hook] \ar[r, hook] & B_1 \ast B_2 \ar[u, hook, "\pi"]
	\end{tikzcd}
	\]
\end{theorem}

\begin{proof}
If the ultrafilter $\U$ is principal, this is Blanchard and Dykema's theorem recalled above. Assume thus that $\U$ is nonprincipal.
	
Let's simply regard $B_i$ as a subalgebra of $\prod_\U A_i^{(n)}$, and $\tau_i$ as the restriction of $\rho_i^{\mathcal{U}}$ to $B_i$ for $i = 1,2$.

Choose countably many generators $\{b_{1,j} : j \in J_1'\}$ of $B_1$, then
choose lifts $(b_{1,j}^{(n)})_n$ in $ \prod_{n=1}^\infty A_1^{(n)}$ for each $b_{1,j}$. 
Using that   $\bigoplus_n  A_1^{(n)}$ is separable, extend this collection  
so that it also contains generators of $\bigoplus_n  A_1^{(n)}$. 
We thus get a countably  indexed collection of elements $\{(b_{1,j}^{(n)}):j\in J_1\}$
in $\prod_n A_{1}^{(n)}$ such that when projected  onto  $A_1^{(n)}$ generates $A_1^{(n)}$
for all $n$ and  when mapped to $\prod_\U A_{1}^{(n)}$ generates $B_1$. Choose similarly an indexed collection of elements 
$\{(b_{2,j}^{(n)})_n:j\in J_2\}$ in $\prod_n A_2^{(n)}$ whose projections onto $A_2^{(n)}$, for all $n$, and onto 
the ultraproduct $\prod_\U A_2^{(n)}$ result in generating families for $A_2^{(n)}$ and $B_2$, respectively.

We have 
\[
\begin{tikzcd}
	\{b_{1,j}^{(n)} : j \in J_1\}\ar{r}{s}[swap]{\U}& \{b_{1,j} : j \in J_1\} 		
\end{tikzcd} \hbox{ and }
\begin{tikzcd}
	\{b_{2,j}^{(n)} : j \in J_2\}\ar{r}{s}[swap]{\U}& \{b_{2,j} : j \in J_2\}. 	
\end{tikzcd}	
\]
Hence, by the previous theorem,
\[
\begin{tikzcd}
	\{b_{1,j_1}^{(n)}, b_{2,j_2}^{(n)} : j_1\in J_1, \, j_2 \in J_2\}\ar{r}{s}[swap]{\U}& \{b_{1,j_1}, b_{2,j_2} : j_1\in J_1, \, j_2 \in J_2\}. 	
\end{tikzcd}	
\]
This, in turn, yields 
\[
\pi \colon B_1 \ast B_2 \to \prod_\U (A_1^{(n)} \ast A_2^{(n)}),
\]
that maps $b_{i,j}$, for $i=1,2$, to the element in $\prod_\U (A_1^{(n)} \ast A_2^{(n)})$ with lift  $(b_{i,j}^{(n)})_n$.
This is the desired embedding.
\end{proof}

\begin{corollary}\label{freeexistential}
For $i=1,2$, let $\theta_i\colon (A_i,\rho_i)\to (B_i,\tau_i)$ be existential embeddings of C*-probability spaces, where the C*-algebras are all separable and the states have faithful GNS representations.  Then, the embedding 
\[
\theta_1*\theta_2\colon (A_1, \rho_1) \ast (A_2, \rho_2) \to (B_1, \tau_1) \ast (B_2, \tau_2)
\]
is also existential.
\end{corollary}

\begin{proof}
Choose an ultrafilter $\U$ on $\N$ and  embeddings  $\sigma_i\colon (B_i,\tau_i)\to (A_i^{\U},\rho_i^{\U})$  for $i=1,2$
such that $\sigma_i\theta_i$ agrees with the diagonal embedding of $A_i$ in $A_i^{\U}$.
By the previous theorem, we obtain an embedding
\[
\pi\colon (B_1,\tau_1)*(B_2, \tau_2)\to ((A_1*A_2)^{\U}, (\rho_1*\rho_2)^{\U}).
\]
Consider the  composition $\phi= \pi\circ(\theta_1*\theta_2)$ of the embedding $\theta_1*\theta_2\colon A_1*A_2 \to B_1*B_2$  with $\pi$.
We readily verify that the restrictions of $\phi$
to $A_1$ and $A_2$ (regarded as subalgebras of $A_1*A_2$) coincide with the diagonal embeddings
of these algebras in $(A_1*A_2)^{\U}$. Thus, $\phi$ agrees with the 
diagonal embedding of $A_1*A_2$ in $(A_1*A_2)^{\U}$, which shows that  $\theta_1*\theta_2$ is existential.
\end{proof}

\section{Definition and first examples} \label{definition}

\begin{definition}\label{selflessdef}
Let $(A,\rho)$ be a C*-probability space,  where $\rho$ has faithful GNS representation.  We call $(A,\rho)$ selfless if $A\neq \C$ and 
the  first factor embedding $(A,\rho)\hookrightarrow (A,\rho)*(A,\rho)$  is existential.
\end{definition}

%
%
%

The following proposition provides some concrete examples.  
\begin{proposition}
Let $(A,\rho)$ be a C*-probability space, where $\rho$ has faithful GNS representation. Then $\freeprod_{i=1}^\infty (A,\rho)$ is selfless.
\end{proposition}
\begin{proof}
Let $(B,\tau)=\freeprod_{i=1}^\infty (A,\rho)$ and   $B_n= \freeprod_{i=1}^n A$ for $n=1,2,\ldots$,
which we regard as subalgebras of $B$. 
Call $\theta$ the first factor embedding of $(B,\tau)$ into $(B,\tau)*(B,\tau)$. Notice that  $(B,\tau)\cong (B,\tau)*(B,\tau)$, and that, by re-shuffling the factors of the reduced free products, we can  get isomorphisms  
$\sigma_n\colon B*B\to B$ such that $\sigma_n\theta=\mathrm{id}|_{B_n}$ for all $n$.
It follows by Lemma \ref{unusedsymbol} that $\theta$ is existential, i.e., $(B,\rho)$ is selfless.  (More directly, 
the sequence $(\sigma_n)_n$ gives an embedding $\sigma\colon B*B\to B^{\U}$, with $\U$ a nonprincipal ultrafilter on $\N$, such that  $\sigma\theta$ is the diagonal embedding of $B$ in $B^{\U}$.) 
\end{proof}

Let $\F_\infty$ denote the free group in infinitely many generators. Let $\mathcal O_\infty$ denote
the Cuntz algebra generated by infinitely many isometries. 
The C*-algebras $C_r^*(\mathbb F_\infty)$ and $\mathcal O_\infty$ fit the template of the previous proposition:
\begin{align}\label{finftyoinfty}
(C_r^*(F_\infty),\tau) &\cong *_{n=1}^\infty (C_r^*(\Z),\lambda),\\
(\mathcal O_\infty, \phi) &\cong  *_{n=1}^\infty (C_r^*(\N),\delta_0),
\end{align}
where $\delta_0$ is the state on $C_r^*(\N)$ such that  $\delta_0((s_1)^n)=0$ for all $n\geq 1$. (See Avitzour, \cite[\S 2.5]{voiculescu}).
We thus obtain the following:
\begin{corollary}\label{corofinftyoinfty}
$(C_r^*(\mathbb F_{\infty}),\tau)$ and $(\mathcal O_\infty,\phi)$ are selfless.
\end{corollary}
Since the state $\phi$ is pure, and by the homogeneity of the set of pure states in a simple separable C*-algebra, 
we get that $(\mathcal O_\infty,\psi)$ is selfless for any pure state $\psi$. (Cf. Theorem \ref{selflesskirchberg} below.)

\begin{lemma}\label{subselfless}
If $(B,\tau)$ is selfless and $\theta\colon (A,\rho)\hookrightarrow (B,\tau)$ is an existential embedding, then $(A,\rho)$
is selfless.
\end{lemma}

\begin{proof}
Note that $A\neq \C$, since 
		$B$ embeds in an ultrapower of $A$ and $B\neq \C$.
		Let us show next  that $\rho$ has faithful GNS representation. Let $a\in A$ be positive and nonzero. Since $\tau$ has faithful GNS representation,
		there exist $\delta>0$ and $y\in B$, in the unit ball, such that $\tau(y\theta(a)y^*)>\delta$. Since $\theta$
		is existential, we find $x\in A$ such that $\rho(xax^*)>\delta$. Thus, $\rho$ has faithful GNS representation.

	Let $\psi\colon B \to B*B$ be the first factor embedding, which is  existential.
	The composition   $\psi \circ \theta$ is thus  existential, by Lemma \ref{lemexistential}. We have the commutative diagram 
	\[
	\begin{tikzcd}
		(B,\tau) \ar[r, hook] & (B,\tau)*(B,\tau) \\
		(A, \rho) \ar[r, hook]\ar[u, hook, "\theta"] & (A, \rho)*(A, \rho)\ar[u, hook, "\theta*\theta"]
	\end{tikzcd},
	\]
	where the horizontal arrows are the first-factor embeddings. Since the first factor embedding   $A\hookrightarrow A*A$ composed with $\theta*\theta$ yields $\psi\theta$, which  we have shown is existential, we get that $A\hookrightarrow A*A$ is existential (Lemma \ref{lemexistential}), as desired.
\end{proof}

\emph{Note}: When we speak below of freely independent elements $a_1,\ldots,a_n$ in $(A,\rho)$, we understand by it free independence of the C*-subalgebras that they each generate.

Recall that a unitary $u\in A$ is called a Haar unitary, relative to a state $\rho$, if $\rho(u^n)=0$ for all $n\neq 0$; equivalently, if the restriction of $\rho$ to $C^*(u)$ is represented by the normalized Lebesgue measure on $\T$.

\begin{lemma}\label{superconvergence}
Let $(A,\rho)$ be a C*-probability space such that  $A\neq \C$. Then $\freeprod_{k=1}^n (A,\rho)$ contains a Haar unitary for large enough $n$. 
\end{lemma}

\begin{proof}
This is  a consequence of Bercovici and Voiculescu’s results
on superconvergence in the free central limit theorem \cite{bercovici-voiculescu}. We follow the argument used in \cite[Claim 2]{choda-dykema}.

Choose $a\in A$,  selfadjoint element of norm 1 such that $\rho(a)=0$ (guaranteed to exist since $A\neq \C$). Let $\mu$
be its distribution with respect to $\rho$. By \cite[Proposition 8]{bercovici-voiculescu}, for a large enough $n$ the free convolution $\nu= \boxplus_{k=1}^n\mu$ is absolutely continuous with respect to the Lebesgue measure and has support equal to an interval $[\alpha,\beta]$. For $k=1,\ldots,n$, let $\theta_k\colon A\hookrightarrow \freeprod_{i=1}^n A$ denote the embedding into the $k$-th factor,  and set $a_k=\theta_k(a)$. Then $a_1,a_2,\ldots,a_n$ are freely independent and have distribution $\mu$. Hence, the distribution of  $b=a_1+a_2+\ldots+a_n$ is the  measure $\nu$. It follows that $C^*(1,b)$ contains a Haar unitary, by \cite[Proposition 4.1 (i)]{d-h-r}.
\end{proof}

\begin{theorem}\label{selflessunitary}
Let $(A,\rho)$ be a C*-probability space, with  $A\neq \C$ and  $\rho$ state with faithful GNS.  The following are equivalent:
\begin{enumerate}[(i)]
\item
$(A,\rho)$ is selfless.

\item
The first factor embedding 
\[
(A, \rho) \hookrightarrow (A, \rho) \ast (C, \kappa)
\] 
is existential, for any   $(C,\kappa)$, where $C$ is separable, $\kappa$ has faithful GNS representation, 
and $(C,\kappa)$ embeds in $(A^\U,\rho^\U)$ for some ultrafilter $\U$.

\item
The first factor embedding 
\[
(A, \rho) \hookrightarrow (A, \rho) \ast (C, \kappa)
\] 
is existential, for some  $C\neq \C$ and state $\kappa$  with faithful GNS representation.

\item
The first factor embedding 
\[
(A,\rho)\hookrightarrow (A,\rho)*(C(\T),\lambda)
\] 
is existential,  where $\lambda$ is the trace induced by the normalized Lebesgue measure on $\T$.

\item
The first factor embedding 
\[
(A,\rho)\hookrightarrow (A,\rho)*(C_r^*(\F_\infty),\tau)
\] 
is existential.

\item
The first factor embedding 
\[
(A,\rho)\to \freeprod_{i=1}^\infty (A,\rho)
\] 
is existential.
\end{enumerate}
\end{theorem}

\begin{proof}
We shall use repeatedly that if an existential embedding $\theta$ can be factored as $\theta=\theta_2\theta_1$, where
$\theta_1$ and $\theta_2$ are embeddings, then $\theta_1$ is  existential (Lemma \ref{lemexistential}).
We shall use the phrase  ``$\theta$ factors through $\theta_1$" to describe this situation (with the understanding that $\theta_2$
is also an embedding).


(i) implies (ii). 
Suppose first that $A$ is separable. Let $\U$ be an ultrafilter on $\N$. 
From the embeddings $A\hookrightarrow A^\U$ and $(C,\kappa)\hookrightarrow (A^\U,\rho^\U)$
we get, by Theorem \ref{freeexactness}, an embedding $(A,\rho)*(C,\kappa)\hookrightarrow  ((A*A)^\U, (\rho*\rho)^\U)$ such that the following diagram commutes
\[
\begin{tikzcd}
A^\U \ar[r, hook] & (A*A)^\U\\
A \ar[r, hook]\ar[u, hook] & A*C\ar[u, hook]
\end{tikzcd}.
\]
The  top horizontal arrow   $\theta^\U\colon A^{\U}\hookrightarrow (A*A)^\U$ is the ultrapower of the first factor embedding $\theta\colon A\to A*A$. It is existential by Lemma \ref{ultraexist}.
Thus, the  composition of the diagonal  embedding   $A\hookrightarrow A^\U$   with $\theta^\U$ is existential. Since this composition factors through  $A\hookrightarrow A*C$, the latter  is also existential.
 
To remove the assumption of separability, we argue in the standard way by expressing $(A,\rho)$ as an inductive limit of separable elementary submodels. Let $(A',\rho') \subseteq (A,\rho)$ be a separable  elementary submodel (in the theory of C*-probability spaces). 
Then $(A',\rho')$ is again selfless (Lemma \ref{subselfless}). Since $(A,\rho)$ and $(A',\rho')$ are elementarily equivalent, and $C$ is separable, 
$(C,\kappa)$ embeds in $((A')^\U,(\rho')^\U)$. Thus, the embedding $A'\hookrightarrow A'*C$ is existential, by the already established separable case of (ii). By the downward L\"{o}wenheim-Skolem theorem \cite[Theorem 2.6.2]{Cstarmodel}, $(A,\rho)$ is the direct limit of its separable elementary submodels.  It follows by Lemma \ref{unusedsymbol},
 applied to the upward directed family of separable elementary submodels of $A$, that $A\hookrightarrow A*C$ is existential.

(ii) implies (iii). Choose $C=A$.

(iii) implies (iv). Choose $C'\subseteq C$ separable and such that $C'\neq \C$. The embedding  $A\hookrightarrow A*C$ factors into embeddings
 $A\hookrightarrow A*C'$ and $A*C'\hookrightarrow A*C$, so $A\hookrightarrow A*C'$ is existential. Thus, we may assume without loss of generality that $C$ is separable.

Assume first that $A$ is separable. By Corollary \ref{freeexistential}, the embedding 
\[
(A, \rho) \ast (C, \kappa) \hookrightarrow (A, \rho) \ast (C, \kappa) \ast (C, \kappa)
\]
is existential. Since the composition of existential embeddings is existential (Lemma \ref{lemexistential}),  
\[
(A, \rho) \hookrightarrow (A, \rho) \ast (C, \kappa) \ast (C, \kappa)
\]
is existential. Repeating this argument, we obtain that  
\[
(A, \kappa) \hookrightarrow (A, \rho) \ast \freeprod_{i=1}^n (C, \kappa)
\]
is existential for all $n$. By Lemma \ref{superconvergence}, for some $n \in \mathbb{N}$, the reduced 
free product $\freeprod_{i=1}^n (C, \kappa)$  contains a Haar unitary. This implies that  the first factor embedding of $(A, \rho)$ into 
\[
(A, \rho) \ast \freeprod_{i=1}^n (C, \kappa)
\]
factors through the embedding 
\[
(A, \rho) \to (A, \rho) \ast (C(\mathbb{T}), \lambda).
\]
Thus the latter embedding is existential. We have thus proven (iv) assuming that $A$ is separable. 

For a general C*-algebra $A$, we argue as before: let $(A',\rho') \subseteq (A,\rho)$ be a separable  elementary submodel in the theory of C*-probability spaces. 
The embedding  $(A', \rho')\hookrightarrow (A, \rho) \ast (C, \kappa)$ is existential, as it is the composition of two existential embeddings. 
Since it factors through $(A', \rho|_{A'}) \hookrightarrow (A', \rho|_{A'}) \ast (C, \kappa)$, the latter is  
 existential. Consequently, 
 \[
 (A', \rho')\hookrightarrow   (A', \rho') \ast (C(\T), \lambda)
 \] 
 is existential. Passing to the direct limit over all separable elementary submodels of $(A,\rho)$, and using Lemma \ref{unusedsymbol},  we obtain  that  $(A, \rho) \to (A, \rho) \ast (C(\T), \lambda)$  is existential, as desired.

(iv) implies (v). Passing to elementary submodels  as before, we may reduce  to the case that $A$ is separable. Taking the reduced  free product with $(C(\T),\lambda)$ 
in the existential embedding $A\hookrightarrow A*C(\T)$
we get that 
\[
(A,\rho)\hookrightarrow   (A, \rho) \ast (C_r^*(\mathbb F_2), \lambda*\lambda)
\] 
is existential. But 
$C_r^*(\mathbb F_\infty)$ embeds in  $C_r^*(\mathbb F_2)$. Thus, arguing as before, the embedding 
$(A,\rho)\hookrightarrow  (A, \rho) \ast (C_r^*(\mathbb F_\infty), \tau)$ is existential.


(v) implies (vi).
Let $u\in C^*(\T)$ be the generator  Haar unitary, i.e., the  identity on $\T$. Then  $A_k=u^kAu^{-k}$, for $k\in \Z$, 
are freely independent C*-subalgebras of $(A*C(\T),\rho*\lambda)$. As shown in the proof of  \cite[Lemma 4.1]{dykema-sh},
the restriction of $\rho*\lambda$ to  $C^*(A_k:k\in \Z)$ is a state inducing a faithful GNS representation, and in fact 
\[
C^*(A_k:k\in \Z)\cong \freeprod_{k\in \Z} (A_k,(\rho*\lambda)|_{A_k})\cong  \freeprod_{k\in \Z} (A,\rho).
\]
Thus, there exists an embedding of  $\freeprod_{k=1}^\infty(A,\rho)$ into  $(A,\rho)*(C(\T),\lambda)$ agreeing with $\mathrm{Ad}_{u^k}$
on the $k$-th factor of the free product $\freeprod_{k=1}^\infty(A,\rho)$. This shows that we can factor 
$A\hookrightarrow A*C(\T)$ through the first factor embedding $A\hookrightarrow \freeprod_{k=0}^\infty A$,
proving (vi).

(vi) implies (i) is obvious.
\end{proof}

\begin{remark}
The characterization of selflessness in Theorem \ref{selflessunitary} (iv) can be regarded as a C*-version of Popa's theorem asserting the existence, in a tracial ultrapower of a separable II$_1$ factor, of a 
Haar unitary  that is freely independent with the  diagonal copy of the factor \cite{popa}. Popa's theorem has, in part, motivated our definition of selfless C*-probability space. Notice, however, that Theorem \ref{selflessunitary} (iv) asserts more than the existence of a Haar unitary in $A^{\U}$ freely independent from $A$, as it amounts to asking that
\begin{enumerate}
\item
there exists a Haar unitary $u\in A^\U$ freely independent from the diagonal copy of $A$ in $A^\U$, with respect to the limit state $\rho^\U$,
\item
the  restriction of  $\rho^\U$  to $C^*(A,u)\subseteq A^\U$ is a state inducing a faithful GNS representation.
\end{enumerate}
\end{remark}

\begin{theorem}\label{unusedsymbolthm}
Let $I$ be an infinite index set.
Let $(A_i,\rho_i)_{i\in I}$  be C*-probability spaces with $\rho_i$ inducing a faithful GNS representation for all $i\in I$. 
Suppose that for infinitely many $i$ the state $\rho_i$ vanishes on some unitary of $A_i$. Then $(A,\rho)=\freeprod_{i\in I} (A_i,\rho_i)$ 
is selfless.
\end{theorem}

\begin{proof} 
Let $\theta \colon A\to (A,\rho)*(C(\T),\lambda)$ be the first factor embedding. It will suffice to show that $\theta$ is an existential embedding, by the previous theorem. 

For each finite set $F\subset I$, let 
\[
(A_F,\rho|_{A_F})=\freeprod_{i\in F}(A_{i},\rho_i)
\]
(where we regard $A_F$ as a C*-subalgebra of $A$).  
Then $(A_F)_{F}$ is an upward directed family of C*-subalgebras of $A$, indexed by the the finite subsets of $I$, with dense union in $A$.
The C*-algebras $A_F*C(\T)$ also form an upward directed family of C*-subalgebras of $A*C(\T)$, with dense union. Clearly, 
$\theta(A_F)\subseteq A_F*C(\T)$. By Lemma \ref{unusedsymbol}, to show that $\theta$ is existential it will suffice to show that $\theta|_{A_F}$ is relatively existential in $A$ for all $F$. We show this next.

Fix a finite set $F\subset I$. Find distinct indices $i,j\in I\backslash F$ such that $A_i$ and $A_j$ contain unitaries $u_i$ and $u_j$, respectively, on which the states vanish. Since they are freely independent, it is easily checked that  $u=u_iu_j\in A_i*A_j$ is a Haar unitary. 
Let $\phi\colon (C(\T),\lambda)\hookrightarrow (A_i*A_j,\rho_i*\rho_j)$ be the embedding induced by $u$. 
By Blanchard and Dykema's Theorem \ref{thm-blanchard-dykema}, there exists an embedding $\sigma_F$ of  $A_F*C(\T)$ in $A_F*A_i*A_{j}\subseteq A$ such that
$\sigma_F\circ (\theta|_{A_F})$ agrees with the inclusion of $A_F$ in $A$.  In particular,   
$\theta|_{A_F}\colon A_F\to A_F*C(\T)$ is relatively existential in $A$, as desired.
\end{proof}

\section{Purely infinite/stably finite dichotomy}\label{dichotomy_sec}

Let $(A,\rho)$ be a C*-probability space, with $\rho$ inducing a faithful GNS representation.
Let us say that $A$ has the uniform Dixmier property with respect to $\rho$ if there exist $N\in \N$ 
and $0<\gamma<1$ such that for any $c\in A$, with $\rho(c)=0$, there exist unitaries $u_1,\ldots,u_N\in A$ such that
we have
\[
\Big\|\frac1N\sum_{i=1}^N u_icu_i^*\Big\|\leq\gamma \|c\|. 
\]
It is well known that this property implies that $A$ is a simple C*-algebra,
and that if $\rho$ is not a trace, then $A$ is traceless, while if $\rho$ is a trace, then it is the unique
tracial state of $A$; see for example the last three paragraphs of the proof of \cite[Proposition 3.2]{dykema-somefree}. The uniform Dixmier property with respect to $\rho$ for $(N,\gamma)$ is $\forall\exists$-axiomatizable
in the language of C*-probability spaces \cite[Lemma 7.2.2]{Cstarmodel}. For a more general version of the uniform Dixmier property, see \cite[Definition 3.1]{ART}.

Let us recall the definition  of the property of strict comparison of positive elements by  traces. We will restrict ourselves to the case of a simple unital C*-algebra with a unique trace, though this property can be defined much more generally (see \cite[Proposition 6.2]{ERS}). Let $A$ be a simple unital C*-algebra  with a unique tracial state $\rho$. We say that $A$ has the property  of strict comparison of positive elements by $\rho$ if 
for any two positive elements  $a,b\in A\otimes\mathcal K$
\[
d_\rho(a)<d_\rho(b)\Rightarrow a\precsim b,
\]
where $d_\rho(c):=\lim_n \rho(c^{\frac1n})$ and  $\precsim$ denotes the Cuntz comparison relation.

\begin{theorem}\label{dichotomy}
Let $(A,\rho)$ be a selfless C*-probability space (where $\rho$ has faithful GNS representation). 
Then $(A,\rho)$ has the uniform Dixmier property  with respect to $\rho$. In particular,
$A$ is a simple C*-algebra that is either traceless, if $\rho$ is not a trace, or such that $\rho$ is the unique tracial state on $A$. 
In addition, the following are true:
\begin{enumerate}[(i)]
\item
If $\rho$ is faithful and not a trace, then $A$ is purely infinite.

\item
If $\rho$ is a trace (necessarily faithful), then $A$  has  stable rank one and strict comparison of positive elements by $\rho$, 
and $\rho$ is the unique 2-quasitracial state on $A$.
\end{enumerate}
\end{theorem}
\begin{proof}
Let $(B,\rho*\tau)=(A,\rho)*(C_r^*(\mathbb F_\infty),\tau)$, and let $\theta\colon A\to B$ the first factor embedding, which we know 
is existential.
We will use repeatedly that if a property of C*-algebras or C*-probability spaces is $\forall\exists$-axiomatizable, and it is satisfied by $(B,\rho*\tau)$, then it is also satisfied by $(A,\rho)$, 
owing to the fact that the embedding $\theta$ is existential.	

By the proof of \cite[Theorem 2]{dykema} or \cite[Lemma 3.0]{avitzour},  
$B$ has the uniform Dixmier property with respect to $\rho*\tau$ with $N=5$ and $\gamma=2/\sqrt 5$. Since the $(N,\gamma)$ uniform Dixmier property is $\forall\exists$-axiomatizable in the  language of C*-probability spaces, it follows that  $A$ has the  uniform Dixmier property with respect to $\rho$.

(i)  If $\rho$ is faithful and not tracial, then Dykema and R{\o}rdam show in  \cite[Theorem 2.1]{dykema-rordamII} that  $B$  is purely infinite. 
Since being simple and purely infinite is $\forall\exists$-axiomatizable in the language of C*-algebras \cite[3.13.7]{Cstarmodel}, it follows that
 $A$ is purely infinite.

(ii)  By \cite[Theorem 3.8]{d-h-r}, $B$ has stable rank one. Since the stable rank one property is $\forall\exists$-axiomatizable
\cite[Proposition 3.8.1]{Cstarmodel}, $A$ has stable rank one as well.


By \cite[Proposition 6.3.2]{nccw}, $B$ has strict comparison  of positive elements by the trace $\rho*\tau$. Since the class
of C*-algebras with strict comparison by traces is $\forall\exists$-axiomatizable \cite[Theorem 8.2.2]{Cstarmodel}, it follows that $A$ has strict comparison by (its unique trace) $\rho$.

By  \cite[Theorem 3.6]{ng-robert}, $\rho*\tau$ is the unique 2-quasitracial state on $B$.
To see that $\tau$ is the unique 2-quasitrace on $A$, choose an embedding  $\sigma\colon B\to A^\U$   such that $\sigma\theta$ agrees with the diagonal embedding of $A$ in $A^\U$. Let $\phi$ be a 2-quasitracial state on $A$. 
For $a\in A_\U$ with lift $(a_i)_i\in \prod_i A$, define
$\phi_\U(a)=\lim_\U \phi(a_i)$. This defines a 2-quasitracial state on $A^\U$. 
Since $\rho*\tau$ is the unique 2-quasitracial state on $B$, we must have that $\phi_\U\sigma=\rho*\tau$. 
Thus, 
\[
\phi=\phi_\U\sigma\theta=(\rho*\tau)\theta=\rho.\qedhere
\] 
\end{proof}


\section{Permanence properties}\label{permanence}

\begin{theorem}\label{directlimits}
Let $(A,\rho)$ be a C*-probability space. Suppose that $A=\overline{\bigcup_{i\in I} A_i}$, where $(A_{i})_{i\in I}$ is an upward directed family  of unital C*-subalgebras of $A$.  If  $(A_i,\rho|_{A_i})$ is selfless for all $i$, then $(A,\rho)$ is selfless.
\end{theorem}

\begin{proof} 
Clearly, $A\neq \C$,  since $A_i\neq \C$ for all $i$. Since each $A_i$ is a simple C*-algebra (Theorem \ref{dichotomy}), $A$
is simple, and in particular $\rho$ has faithful GNS representation.
	
Let $\theta\colon A\to A*A$  be the first factor embedding. The C*-algebras $(A_i*A_i)_i$ form an upward directed family whose direct limit is $A*A$. Moreover, $\theta(A_i)\subseteq A_i*A_i$, and $\theta|_{A_i}$, regarded as a map with codomain $A_i*A_i$, agrees with the first factor embedding of $A_i$ in $A_i*A_i$. 
Since $(A_i,\rho|_{A_i})$ is selfless,  the embedding $\theta|_{A_i}$ is existential, and in particular,  existential relative to $A$ (in  the sense of Lemma \ref{unusedsymbol}) for all $i$. It follows by Lemma \ref{unusedsymbol} that $\theta$ is existential.
\end{proof}

%
%
%
%

\begin{theorem}\label{selflessprod}
Let $(A, \rho)$ be a selfless C*-probability space. Let $(B, \tau)$ be a separable C*-probability space, where $\tau$ has faithful GNS representation. 
Then, the reduced free product $(A, \rho) * (B,\tau)$ is again selfless.
\end{theorem}

\begin{proof}
Assume first that $A$ is separable. Consider the existential embedding 
\[
(A,\rho) \to (A, \rho) \ast (C(\mathbb{T}), \lambda).
\]
By taking the reduced free product with $(B, \tau)$, we obtain an embedding 
\[
(A, \rho) \ast (B, \tau) \to (A, \rho) \ast (B, \tau) \ast (C(\mathbb{T}), \lambda),
\]
which is  existential by Corollary \ref{freeexistential}. Clearly, $A*B\neq \C$ and the reduced free product state
$\rho*\tau$ has faithful GNS representation. Thus,  $(A, \rho) \ast (B, \tau)$ is selfless, by Theorem \ref{selflessunitary}.

Let us drop the assumption that $A$ is separable. Regard $(A,\rho)$ as the direct limit of its separable elementary 
submodels. Each such  $(A',\rho')$ is selfless (by Lemma \ref{subselfless}). Hence  $A' \ast B$  is selfless as well. The C*-algebras
$A'*B$ form a directed system of selfless C*-algebras (relative to the restriction of $\rho \ast \tau$ to $A' \ast B$) with 
limit   $A \ast B$. It follows that $A \ast B$ is selfless, by Theorem \ref{selflessunitary}.
\end{proof}

Let us recall the definition of semicircular and circular elements of radius 1: 
A selfadjoint element $s\in A$ is called semicircular (of radius 1) if $s=\frac12(u+u^*)$, where $u$ is a Haar unitary (relative to a state $\tau$).  An element $x\in A$ is called circular (of radius 1) if $x=\sqrt 2 a+i\sqrt 2 b$, where $a,b$ are semicircular and freely independent.

\begin{theorem}\label{matrices}
Let $(A,\rho)$ be a selfless C*-probability space, with $\rho$ a faithful state. Then $(M_n(A),\rho\otimes \mathrm{tr}_n)$ is selfless for all $n\in \N$
(where $\mathrm{tr}_n$ denotes the normalized trace on $M_n(\C)$).
\end{theorem}

\begin{proof}
Let $\theta$ denote the first factor  embedding of $(A,\rho)$ into
$(A,\rho)*(C_r^*(\F_\infty),\tau)$. 
 Since $(A,\rho)$ is selfless,  there exists an embedding 
 $\sigma\colon A*C_r^*(\F_\infty)\hookrightarrow A^{\U}$ 
such that $\sigma\theta$ agrees with the diagonal embedding of $A$ in $A^\U$ (Theorem \ref{selflessunitary}). Tensoring with $M_n(\C)$ and identifying $A\otimes M_n(\C)$ and $A^\U\otimes M_n(\C)$ with $M_n(A)$ and $M_n(A)^\U$, respectively, 
we obtain that the embedding of   $M_n(A)$ in $M_n(A*C_r^*(\F_\infty))$ is existential:
\[
\begin{tikzcd}
	(M_n(A),\rho\otimes \mathrm{tr}_n) \ar[r, hook, "\iota"] 
	\ar[rd, hook] &  (M_n(A)^{\mathcal{U}},(\rho\otimes \mathrm{tr}_n)^{\mathcal{U}}) \\
	& (M_n(A * C_r^*(\mathbb{F}_{\infty})), (\rho * \tau)\otimes \mathrm{tr}_n) \arrow[u, hook] 
\end{tikzcd}
\]
(What we have shown is that
augmenting an existential embedding of C*-probability spaces to $n\times n$ matrices  results in an existential embedding.)

Let us show  that the embedding $M_n(A)\hookrightarrow M_n(A*C_r^*(\F_\infty))$
can be  factored through the first factor embedding of $(M_n(A),\rho\otimes \mathrm{tr}_n)$  
in $(M_n(A),\rho\otimes \mathrm{tr}_n)*(C(\T),\lambda)$. The argument is inspired by the proof of
\cite[Proposition 3.3]{dykema-rordamGAFA}.

Choose in $C_r^*(\F_\infty)$ a collection of freely independent elements $(x_{ij})_{ij=1}^n$, with $1\leq i\leq j\leq n$,  such that $x_{ii}$ is semicircular for all $i$ and $x_{ij}$ is circular for all $i<j$. This is clearly possible from the availability of infinitely many freely independent Haar unitaries in $C_r^*(\F_\infty)$. Set
\[
x=\frac{1}{\sqrt n}\Big(\sum_{i=1}^n x_{ii}\otimes e_{ii} + \sum_{1\leq i<j\leq n} x_{ij}\otimes e_{ij} + x_{ij}^*\otimes e_{ji}\Big).
\]
By \cite[Theorem 2.1]{dykema-rordamGAFA}, $x$ is a semicircular element and $M_n(A)$ and $C^*(x,1)$ are freely independent C*-subalgebras  of $M_n(A*C_r^*(\F_\infty))$ relative to  $(\rho*\tau)\otimes \mathrm{tr}_n$ in  $M_n(A*C_r^*(\F_\infty))$. Using functional calculus, we obtain a Haar unitary  $u\in C^*(x,1)$. We thus have freely independent embeddings of  $M_n(A)$ and $C(\T)$ into  $M_n(A*C_r^*(\F_\infty))$. Since 
$(\rho*\tau)\otimes \mathrm{tr}_n$ is a faithful state \cite{dykema-faithful}, there exists  (by \cite[Lemma 1.3]{dykema-rordamGAFA}) an embedding 
\[
\sigma'\colon M_n(A)*C(\T)\to M_n(A*C_r^*(\F_\infty))
\] 
mapping $M_n(A)$ to $M_n(A)$. Thus, the embedding of $M_n(A)$ in $M_n(A*C_r^*(\mathbb F_\infty))$ factors through the first factor embedding of 
$M_n(A)$ in $M_n(A)*C(\T)$. Since, as shown above, the former embedding is existential, so is the latter. Thus $(M_n(A),\rho\otimes \mathrm{tr}_n)$ is selfless by Theorem \ref{selflessunitary}.
\end{proof}

\begin{theorem}
Let $(A,\tau)$ be selfless, with $\tau$ a faithful trace. Let $p\in A$ be a nonzero projection. Then 
$(pAp,\frac1{\tau(p)}\tau)$ is selfless. 
\end{theorem}

\begin{proof}
Set $(B,\bar \tau)=\freeprod_{i=1}^\infty (A,\tau)$ and denote by $\theta\colon A\to A*B$ the first factor embedding. By assumption, there exists an embedding $\sigma\colon A*B\to A^\U$ such that $\sigma\theta$ agrees with the diagonal embedding of $A$ in $A^\U$. Notice that $\sigma$ maps $\theta(p)\in A*B$ to $p\in A\subseteq A^\U$. Thus, $\sigma$ maps $\theta(p)(A*B)\theta(p)$ into $p(A^\U)p=(pAp)^\U$. It  follows that $\theta|_{pAp}\colon pAp\to \theta(p)(A*B)\theta(p)$ is an existential embedding (cf. \cite[Theorem 2.8]{barlak-szabo}). 

To simplify notation, let us denote $\theta(p)$ simply by $p$, regarded now as an element in $A*B$. 
From Nica and Speicher's \cite[Application 1.13]{nica-speicher}, we know that $pAp$ and $pBp$ are freely independent unital C*-subalgebras of $p(A*B)p$ relative to $\bar\tau_p:=\frac{1}{\tau(p)}(\tau*\bar\tau)$ (the free product trace normalized at $p$). 
Since $\bar\tau_p$ is a faithful trace on $p(A*B)p$, there exists,  by \cite[Lemma 1.3]{dykema-rordamGAFA}, an embedding
\[
(pAp,\frac{1}{\tau(p)}\tau)*(pBp,\bar\tau_p|_{pBp})\stackrel{\sigma'}{\longrightarrow} (p(A*B)p,\bar\tau_p)
\] 
such that $\theta|_{pAp}=\sigma'\theta'$, where $\theta'\colon pAp\to (pAp)*(pBp)$ denotes the first factor embedding of $pAp$ into $(pAp)*(pBp)$. Since, as remarked above, $\theta|_{pAp}$ is existential, $\theta'$ is existential as well. Note that  $pBp\neq \C$, since $B$ is simple and non-elementary. Thus, $(pAp,\frac1{\tau(p)}\tau)$ is selfless, by Theorem \ref{selflessunitary}.
\end{proof}

The class of selfless C*-probability spaces is in general not closed under ultrapowers. 
Indeed, if $(A,\tau)$ is selfless and tracial, then the trace-kernel ideal of 
$\tau^{\mathcal U}$ is a non-trivial closed two-sided ideal in 
$A^{\mathcal U}$ (recall that $\dim A=\infty$). Thus, $A^{\mathcal U}$ is non-simple, whence also not selfless.
However, we do have the following:

\begin{theorem}
The class of selfless C*-probability spaces $(A,\rho)$ where $A$ is a simple purely infinite C*-algebra is $\forall\exists$-axiomatizable.
\end{theorem}

\begin{proof}
We will show that the class in question is closed under direct limits, elementary submodels, and ultraproducts over the natural numbers, which yields the statement of the theorem, by \cite[Proposition 2.4.4]{Cstarmodel}. Since the class of simple purely infinite unital C*-algebras is $\forall\exists$-axiomatizable \cite[3.13.7]{Cstarmodel}, the property of being simple and purely infinite is preserved under these constructions. It remains check that the same holds for being selfless.

If $(A,\rho)$ is an elementary submodel of a selfless $(B,\tau)$, then the inclusion map is existential, so $(A,\rho)$ is selfless by Lemma \ref{subselfless}.

Suppose we have a direct system of selfless C*-probability spaces $(A_i,\rho_i)_i$, $i\in I$, where each $A_i$ is simple and purely infinite.
The maps between them must be embeddings, as the C*-algebras are simple. It follows by Theorem \ref{directlimits} that their direct limit $(A,\rho)$ is selfless.

Consider an ultraproduct $(A,\rho)=\prod_\U (A_i,\rho_i)$, where $\U$ is an ultrafilter on the natural numbers,  
each $(A_i,\rho_i)$ is selfless, and  each $A_i$ is simple and purely infinite. As remarked above, it follows that $A$ is again simple
and purely infinite.  In particular, the GNS representation induced by $\rho$ has trivial kernel, i.e., it is faithful. 
Let $(B,\rho|_B)$ be a separable elementary submodel of $(A,\rho)$.
Note that $B$ is again simple, and so $\rho|_B$ has faithful GNS representation as well. It will suffice to show that $(B,\rho|_B)$ is selfless,
as we can then pass to the direct limit over all separable elementary submodels of $(A,\rho)$ to reach the same conclusion for $(A,\rho)$.

By Theorem \ref{freeexactness}, we have an embedding 
\[
\pi\colon (B,\rho|_B)*(C(\T),\lambda)\hookrightarrow \prod_\U (A_i,\rho_i)*(C(\T),\lambda)
\]
such that the following diagram commutes:
\[
\begin{tikzcd}
\prod_\U A_i \ar[r, hook] & \prod_\U (A_i*C(\T))\\
B \ar[u,hook]\ar[r, hook] & B*C(\T)\ar[u, hook, "\pi"]
\end{tikzcd}.
\]
The embedding of $B$ into $\prod_\U A_i$ is existential, since it is elementary.
The top horizontal arrow is obtained as the ultraproduct of the existential embeddings $A_i\hookrightarrow A_i*C(\T)$. Thus, it is also existential (Lemma \ref{ultraexist}). It follows that the embedding of $B$ into $B*C(\T)$ is existential, as desired.

\end{proof}



%

\section{Examples} \label{examples}
By a Kirchberg algebra we understand a separable, nuclear,  simple, purely infinite C*-algebra.

\begin{theorem}\label{selflesskirchberg} 
	Let $A$ be a unital Kirchberg algebra in the UCT class and let $\rho$ be a pure state on $A$. Then $(A,\rho)$ is selfless.
\end{theorem}

\begin{proof}
	Let  $(B,\tau)=(A,\rho)*(\mathcal O_\infty, \phi)$, where $\phi$ is the pure state in \eqref{finftyoinfty}. Then $(B,\tau)$ is selfless,  by 
	Corollary \ref{corofinftyoinfty} and Theorem \ref{selflessprod}. Since the reduced free product of pure states is pure (by \cite[Theorem 1.6.5]{dykema-nica-voiculescu}), $\tau$ is pure.
	
	Let us show that $B\cong A$. Let 
	\[
	(A_1,\rho_1)=(A,\rho)*(C_r^*(\N),\delta_0).
	\]
	Then $A_1$ is isomorphic to the Cuntz-Pimsner algebra $\mathcal O_E$ associated to the $A$-Hilbert bimodule  $E=H_\rho\otimes A$, where  $\pi_\rho\colon A\to B(H_\rho)$  is the GNS representation obtained from $\rho$.  See \cite[Proposition 3.4]{kumjian} and \cite[Theorem 2.3]{shlyakhtenko}.  By \cite[Theorem 3.1]{kumjian}, $A_1$ is a Kirchberg algebra  and the embedding $A\hookrightarrow A_1$ is a KK-equivalence.
	It follows that $A_1$ is in the UCT class and that the inclusion induces 
	an isomorphism in K-theory.  Continue defining $(A_{n+1}, \rho_{n+1})=(A_n,\rho_n)*(C_r^*(\N),\delta_0)$ for  $n=1,2,\ldots$. Then each $A_n$ is a Kirchberg algebra in the UCT class and the embedding $A_n\hookrightarrow A_{n+1}$ is an isomorphism in K-theory. It follows that $B=\overline{\bigcup_n A_n}$ is a Kirchberg algebra in the UCT class and the inclusion $A\hookrightarrow B$ induces an isomorphism  $(K_*(A),[1])\cong (K_*(B),[1])$. By the Kirchberg-Phillips classification theorem \cite[Theorem 8.4.1]{rordamencyc}, $A\cong B$.
	
	Through the isomorphism of $A$ and $B$ we obtain  that $(A,\rho')$ is selfless for some pure state $\rho'$. By the homogeneity of the set of pure states of a simple separable C*-algebra \cite{k-o-s}, $(A,\rho)$ is selfless.
\end{proof}

We have already encountered examples of tracial selfless C*-algebras, e.g., $C_r^*(\mathbb F_\infty)$.

\begin{theorem}
The following simple C*-algebras are selfless relative to their unique tracial state:
\begin{enumerate}[(i)]
\item
The Jiang-Su algebra $\mathcal Z$.

\item
The infinite dimensional  UHF C*-algebras.

\item
The tracial ultrapower of a separable $\mathrm{II}_1$ factor over a free ultrafilter on $\N$.
\end{enumerate}
\end{theorem}

\begin{proof}
(i) By \cite[Proposition 6.3.1]{nccw}, there is a (unital) embedding of  $\theta \colon \mathcal Z \to C_r^*(\F_\infty)$. 
 On the other hand, Ozawa has shown in \cite[Theorem 4.1]{ozawa} that there is an embedding $\sigma\colon C_r^*(\F_\infty)\to \mathcal Z^{\U}$. Composing these embeddings, we obtain   $\sigma\theta$, embedding of $\mathcal Z$ in $\mathcal Z^\U$.  Since all unital *-homomorphisms 
from $\mathcal Z$ into $\mathcal Z^\U$ are unitarily equivalent \cite[Theorem 2]{FHRT}, there exists a unitary $u\in \mathcal Z^\U$ such that $(\mathrm{Ad}_u\sigma)\theta$ agrees with the  diagonal embedding of $\mathcal Z$ in $\mathcal Z^\U$. Note that since $\mathcal Z$, $C_r^*(\mathbb F_\infty)$, and $\mathcal Z^\U$
are all monotracial, these embedding preserve the respective traces. It follows that $\theta$ is an existential embedding of C*-probability spaces. Since $C_r^*(\F_\infty)$ is selfless, so is $\mathcal Z$, by Lemma \ref{subselfless}.

(ii) Let  $A$ be an infinite dimensional UHF C*-algebra. Write $A\cong \varinjlim_i M_{n_i}(\C)$. Tensoring with $\mathcal Z$
in this inductive limit and using that $A\otimes \mathcal Z\cong A$,   we get that  $A\cong\varinjlim M_{n_i}(\mathcal Z)$. 
It follows that $A$ is selfless by part (i) and Theorems \ref{directlimits} and \ref{matrices}.

(iii) Let $M$ be a separable $\mathrm{II}_1$ factor with tracial ultrapower $M_\omega$. Let $\tau_\omega$ denote the trace on $M_\omega$. Let $A$ be a separable C*-subalgebra of $M_\omega$.

By Popa's theorem \cite{popa}, there exists a Haar unitary $u \in M_\omega$ that is freely independent from $A$. Since $\tau_\omega$ is a faithful trace, there exists an embedding  
\[
\sigma \colon (A, \tau_\omega|_A) \ast (C(\mathbb{T}), \lambda)\to (M_\omega,\tau_\omega)
\]
such that  $\sigma\theta|_A$ agrees  with the inclusion of $A$ in $M_\omega$ (where $\theta\colon A\to A*C(\T)$ is the first factor embedding). Since $M_\omega$ is the direct limit of its separable C*-subalgebras, it follows, by Lemma \ref{unusedsymbol}, that $M_\omega$ is selfless.
\end{proof}

Combining the previous theorem with Theorem \ref{selflessprod} we get:
\begin{corollary}
For any C*-probability space $(A,\rho)$, with $A$ separable and $\rho$ inducing a faithful GNS representation,
the reduced free product  $(A,\rho)*(\mathcal Z,\tau)$ is selfless. 
\end{corollary}

	\begin{question}
		Let $A$ be a simple, separable, unital, nuclear, $\mathcal{Z}$-stable, monotracial C*-algebra that satisfies the UCT. Is $A$ selfless (with respect to its trace)?
	\end{question}

\section{Eigenfree C*-probability spaces}
In \cite{dykema-rordamGAFA}, Dykema and R{\o}rdam call a C*-probability space 
$(A,\tau)$ eigenfree
 if there exists an  endomorphism $\theta\colon A\to A$
and a Haar unitary $u\in A$ such that $\tau\theta=\tau$, and  $\theta(A)$ and $C^*(u)$ are free relative to $\tau$.  
They show in  \cite[Proposition 3.2]{dykema-rordamGAFA} that a reduced free product 
$(A_1,\tau_1)*(A_2,\tau_2)$ satisfying Avitzour's condition is eigenfree. 
In particular, $C_r^*(\mathbb F_n)$ is eigenfree for  $n=2,3,\ldots,\infty$. 

\begin{proposition}
	Let $(A,\rho)$ be  a C*-probability space, with $\rho$ a faithful state.
	Suppose that $(A,\rho)$  is eigenfree relative to an  endomorphism $\theta\colon A\to A$. 
	Let $B$ be the inductive limit of the stationary system $A\stackrel{\theta}{\to}{A}$.
	Let $\tau$ be the state on $B$, projective limit of the state  $\rho$.
	Then $(B,\tau)$ is selfless. 
\end{proposition}

\begin{proof}
	Let $u\in A$ be a Haar unitary such that $\theta(A)$ and $u$ are freely independent.
	For $n=1,2,\ldots$, let $\theta_{n,\infty}\colon A\to B$ denote the inductive limit maps. 
	Define  $B_n=\theta_{n,\infty}(A)$ and $u_n=\theta_{n,\infty}(u)$ for all $n$. 
	Let $\phi_n\colon B_n\hookrightarrow B_n*C(\T)$ 
	denote  the first factor embeddings for all $n$.

	The free independence of $\theta(A)$ and $u$, relative to $\rho$, readily implies that $B_{n}$ and $u_{n+1}$ are freely independent in $B$,
	relative to $\tau$.  Notice also that  $\tau$ is faithful. (Proof: Let $I$ be the 2-sided ideal-kernel  of $\tau$.
	Then $I=\varinjlim I_n$, with $I_n=\theta_{n,\infty}^{-1}(I)$ for $n=1,2,\ldots$. But $I_n=0$ for all $n$, since $\rho$ is faithful and vanishes on $I_n$.  Thus, $I=0$.) Since $(C^*(u_{n+1}),\tau)\cong (C(\T),\lambda)$, there exists an embedding
	\[
	\sigma_n\colon (B_n,\tau|_{B_n})*(C(\T),\lambda)\to (B,\tau)
	\] 
such that 	$\sigma_n\phi_n$ agrees with the inclusion of $B_n$ in $B$.  This shows that  $\phi_n$ is relatively existential in $B$ for all $n$. 
	It follows, by Lemma \ref{unusedsymbol}, that the embedding of $B$ in $B*C(\T)$ is existential, i.e., $(B,\tau)$ is selfless. 
\end{proof}

\section{C*-algebras that embed in $A^{\U}$}\label{ultraembeddings}

Let $(A,\rho)$ be a selfless C*-probability space. Here we consider the class of separable C*-probability spaces $(B,\tau)$
that embed in $(A^{\U},\rho^{\U})$ for some nonprincipal ultrafilter $\U$ on $\N$. We note this class is independent of the choice of $\U$
(\cite[Corollary 4.3.4]{Cstarmodel}).

\begin{theorem}\label{ultraembeddingsthm}
Let $(A, \rho)$ be a selfless C*-probability space, with $\rho$ inducing a faithful GNS representation. 
Let $\{(B_i, \tau_i) : i \in \mathbb{N}\}$ be  separable C*-probability spaces that embed in 
an ultrapower of $(A, \rho)$. Then the reduced free product $\ast_{i \in \mathbb{N}} (B_i, \tau_i)$ embeds in 
an ultrapower  $(A, \rho)$.
\end{theorem}

\begin{proof}
We may assume that $A$ is separable by considering a separable elementary submodel C*-subalgebra $(A', \rho') \subseteq (A,\rho)$, which is also selfless.

By Theorem \ref{selflessunitary}, the first factor embedding $(A,\rho)\hookrightarrow (A,\rho)*(B_1,\tau_1)$ is existential. Taking reduced free product 
with the identity on $(B_2,\tau_2)$, and applying Corollary \ref{freeexistential}, we get that 
\[
(A,\rho)*(B_2,\tau_2)\hookrightarrow (A,\rho)*(B_1,\tau_1)*(B_2,\tau_2)
\]
is existential. But the first factor embedding of $(A,\rho)$ in $(A,\rho)*(B_2,\tau_2)$ is existential (by Theorem \ref{selflessunitary}).
Thus, the embedding 
\[
(A,\rho)\hookrightarrow (A,\rho)*(B_1,\tau_1)*(B_2,\tau_2)
\]
is existential. Continuing in this way we obtain that 
\[
(A,\rho)\hookrightarrow (A,\rho)*\freeprod_{i=1}^n(B_i,\tau_i)
\]
is existential. Since the C*-algebras $\freeprod_{i=1}^n B_i$ form an increasing family with dense union in $\freeprod_{i=1}^\infty B_i$,
the embedding 
\[
(A,\rho)\hookrightarrow (A,\rho)*\freeprod_{i=1}^\infty(B_i,\tau_i)
\]
is existential. In particular, $\freeprod_{i=1}^\infty(B_i,\tau_i)$ embeds in an ultrapower of $(A,\rho)$.
\end{proof}

\begin{remark}
A von Neumann algebra analog of the previous theorem holds for the tracial ultrapower of any separable II$_1$ factor, by Popa's \cite[Theorem 2.1]{popa}.  	
\end{remark}	

Recall that a C*-algebra is called MF if it is separable and embeds in $\mathcal Q^{\U}$ \cite[Definition 3.2.1]{BlackadarKirchberg}, where $\mathcal Q=\bigotimes_{n=1}^\infty M_n(\C)$. Since $\mathcal Q$ is selfless, by Theorem \ref{examples} (ii), we obtain: 
 
\begin{corollary}
A countable reduced free product of MF C*-algebras is MF.
\end{corollary}



\section{Approximation by commutators and by the sum of a normal and a nilpotent element}\label{commutators}

Here we illustrate how the free independence available in a selfless C*-algebra can be exploited to approximate
elements of $A$ by special classes of elements.

Recall that by free independence between elements $a$ and $b$ in $(A,\rho)$ we understand free independence of the C*-subalgebras that they generate.

Given a positive element $a\in A_+$ and state $\tau\colon A\to \C$, define  $d_\tau(a):=\lim_{n}\tau(a^{\frac1n})$.
(We have already introduced this notation in Section \ref{dichotomy_sec} when $\tau$ is a trace; here we extend it to states.)

\begin{lemma}
Let $\tau$ be a faithful state on a unital $C^*$-algebra $A$. Let $p,a \in A$ be free, with $p$ a projection and $a\in A_+$ a positive element. If $\tau(p) + d_\tau(a) < 1$, then
\[
\|p(a - \tau(a))p\| \leq \|a\| \sqrt{\tau(p)}.
\]
\end{lemma}

\begin{proof}
By the free independence of $a$ and $p$, $\tau$ is a trace on $C^*(a,p)$ \cite[Proposition 2.5.3]{dykema-nica-voiculescu}. Also, the restriction of $\tau$ to $C^*(a,p)$ is faithful. Thus, we may assume without loss of generality that $A=C^*(a,p)$ and that $\tau$ is a faithful trace.

Assume first that $a = q$ is a projection. Set $\tau(p) = \alpha$ and $\tau(q) = \beta$. 

Voiculescu's calculation of the free multiplicative convolution of the distributions of $p$ and $q$ in \cite[Example 2.8]{voiculescu-mult} shows that $pqp$ is invertible in $pAp$ and has spectrum (in $pAp$) the interval with endpoints
\[
\alpha + \beta - 2\alpha\beta \pm 2\sqrt{\alpha(1 - \alpha)\beta(1 - \beta)}.
\]
(Here we have used the assumption that $\alpha+\beta<1$.) The spectrum of $p(q - \tau(q))p$ is then this interval shifted to the left by $\beta$.
Thus, $\|p(q - \tau(q))p\|$ is equal to
\[
\max\left\{\alpha - 2\alpha\beta + 2\sqrt{\alpha(1 - \alpha)\beta(1 - \beta)},\ -\left(\alpha - 2\alpha\beta - 2\sqrt{\alpha(1 - \alpha)\beta(1 - \beta)}\right)\right\}.
\]
For a fixed $\alpha$ and $0 \leq \beta \leq 1$, this expression has maximum $\sqrt{\alpha}$.
Thus,
\[
\|p(q - \tau(q))p\| \leq \sqrt{\alpha} = \sqrt{\tau(p)}.
\]

Now consider a general $a \in A_+$ free from $p$, and such that $\tau(p) + d_\tau(a)< 1$. Without loss of generality, assume that $\|a\| \leq 1$.

Let $(\pi_\tau, H_\tau, \xi)$ be the GNS representation induced by $\tau$. Let $M = \pi_\tau(A)''$. Note that $\tau$ extends to a faithful normal tracial state on $M$, and that $a$ and $p$ are freely independent  in $(M,\tau)$ 
(\cite[Proposition 2.5.7]{dykema-nica-voiculescu}). Let $q_t = \chi_{(t,1]}(a)$, for $t\geq 0$, be spectral projections of $a$. Since $\tau(q_t) \leq d_\tau(a)$ for all $t$, we can apply the previous estimate to get
\[
\|p(q_t - \tau(q_t))p\| \leq \sqrt{\tau(p)}
\]
for all $t$. On the other hand,
\[
a - \tau(a) = \int_{0}^{1} (q_t - \tau(q_t)) \, dt,
\]
where the Riemann sums for the integral on the right-hand side converge in norm to the left-hand side. Therefore,
\[
\|p(a - \tau(a))p\| \leq \int_{0}^{1} \|p(q_t - \tau(q_t))p\| \, dt \leq \sqrt{\tau(p)}.\qedhere
\]
\end{proof}

Given $x,y\in A$, let us write $x\approx_\epsilon y$ if $\|x-y\|< \epsilon$.
Given two subsets $X, Y\subseteq A$, let us write $X \subseteq_\epsilon Y$ if for each $x\in X$ there exists $y\in Y$ such that $x\approx_\epsilon y$.

\begin{theorem}
	Let $(A, \rho)$ be a selfless $C^*$-probability space, where the state $\rho$ is faithful. Suppose that $A$ has real rank zero. For each finite set $F \subseteq \ker \rho$ and $\epsilon > 0$, there exists a unitary $u \in A$ such that $F \subseteq_\epsilon [u, A]$. In particular, the set of single commutators $[u, a]$, with $u \in A$ a unitary and $a \in A$, is dense in $\ker \rho$.
\end{theorem}

\begin{proof}
	Let $\theta_1, \theta_2 \colon A \to (A, \rho) * (A, \rho)$ denote the first and second factor embeddings. 	
	Let us also set $(B, \tau) = (A, \rho) * (A, \rho)$. Notice that, since $\rho$ is faithful, so is $\tau$ (\cite{dykema-faithful}).
	
	Assume without loss of generality that the elements of $F$ have norm at most $1$.
	Let $G = \theta_1(F) \subseteq B$. Let $y\in F$ and set $x=\theta_1(y)$. Decompose $x$  as
	\[
	x = (a_+ - a_-) + i(b_+ - b_-),
	\]
	where $a_+, a_-, b_+, b_- \in B_+$ are positive and orthogonal. Since $\tau(x)=0$, we have $\tau(a_+) = \tau(a_-)$ and $\tau(b_+) = \tau(b_-)$.
	Since $\tau$ is faithful, either $a_+$ and $a_-$ are both zero or both nonzero, and similarly for $b_+$ and $b_-$. In either case,
	it is clear that we can choose $\delta > 0$ such that, for all $x \in G$, we have 
	\[
	\max(d_\tau(a_+), d_\tau(a_-), d_\tau(b_+), d_\tau(b_-)) < 1 - \delta.
	\]
	
	Let $\epsilon > 0$. By the simplicity and real rank zero property of $A$, we can find a partition of unity $p_1, \ldots, p_n$ in $A$ consisting of projections such that $\tau(p_i) < \min(\delta, \epsilon^2)$ for all $i$ \cite[Theorem 1.1]{zhang}. Set $q_i = \theta_2(p_i) \in B$ for $i = 1, \ldots, n$.
	
	Let $y\in F$, and set $x=\theta_1(y) \in G$. Write
	\[
	x = \left(a_+ - \tau(a_+)\right) - \left(a_- - \tau(a_-)\right) + i\left(b_+ - \tau(b_+)\right) - i\left(b_- - \tau(b_-)\right).
	\]
	Since $a_+$ and $q_i$ are free, and $\tau(q_i) + d_\tau(a_+) \leq 1+1-\delta< 1$, the previous lemma implies that
	\[
	\|q_i(a_+ - \tau(a_+))q_i\| \leq \|a_+\| \sqrt{\tau(q_i)} < \epsilon.
	\]
	Applying the same estimate to $a_-, b_+, b_-$, we get
	\[
	\|q_i x q_i\| < 4\epsilon.
	\]
	Therefore, the diagonal entries of $x$, represented as a matrix relative to the projections $q_i$, are $< 4\epsilon$. Let $x'$ be the off-diagonal part of $x$, with zero diagonal entries:
	\[
	x' = x - \sum_{i=1}^n q_i x q_i.
	\]
	Then $\|x - x'\| < 4\epsilon$.
		We can express $x'$ as a commutator in a routine fashion. Let $\omega = e^{2\pi i / n}$, and define the unitary
	\[
	u = \sum_{j=1}^n \omega^{j-1} q_j.
	\]
	Then $x' = [u, x'']$, where $x'' = (x''_{ij})_{ij}$, regarded as a matrix relative to the projections $(q_i)_{i=1}^n$, is defined by
	\[
	x''_{ij} = \frac{1}{\omega^{j-1} - \omega^{i-1}} x'_{ij} \quad \text{for } i \neq j, \quad x''_{ii} = 0.
	\]
	
	Let $\sigma \colon (B, \tau) \to (A^\mathcal{U}, \rho^\mathcal{U})$ be an embedding such that $\sigma \theta_1$ is the diagonal embedding of $A$ in $A^\mathcal{U}$. Then, for $y\in F$, 
	 \[
	  y=\sigma(x)\approx_{4\epsilon} \sigma(x') = [\sigma(u), \sigma(x'')].
	 \] 
	 We thus get the desired approximation of elements	of $F$ by commutators in $A^\mathcal{U}$, which readily translates  into an approximation in $A$.
\end{proof}

\begin{theorem}
	Let $(A, \rho)$ be as in the previous theorem.
	The set of elements of the form $s + t$, with $s$ normal and $t$ nilpotent, is dense in $A$.
\end{theorem}

\begin{proof}
	It will suffice to approximate elements in $\ker \rho$ by sums of a selfadjoint and a nilpotent element.
	
	Let $y \in \ker \rho$ and $\epsilon > 0$. Let $x = \theta_1(y)$ be the image of $y$ in $(B, \tau) := (A, \rho) * (A, \rho)$ via the first factor embedding. As in the proof of the previous theorem, choose projections $(q_i)_{i=1}^n$ in $B$ summing  up to 1, freely independent from $x$, and such that 
	$
	\|q_i x q_i\| < 4\epsilon
	$ for all $i$.
	Again, let
	\[
	x' = x - \sum_{i=1}^n q_i x q_i.
	\]
	Then $\|x - x'\| < 4\epsilon$, and the diagonal entries of $x'$ are zero in its matrix representation relative to $(q_i)_{i=1}^n$. Let
	\[
	s =
	\begin{pmatrix}
		0 & x_{12} & \cdots & x_{1n} \\
		x_{12}^* & 0 & \cdots & x_{2n} \\
		\vdots & \vdots & \ddots & \vdots \\
		x_{1n}^* & x_{2n}^* & \cdots & 0
	\end{pmatrix}
	\quad \text{and} \quad t = x' - s.
	\]
	Then $s$ is selfadjoint, $t$ is nilpotent, and $x' = s + t$. This achieves the desired approximation for $x = \theta_1(y)$ in $B$.
	
	Mapping back to $A^\mathcal{U}$ through $\sigma \colon (B, \tau) \to (A^\mathcal{U}, \rho^\mathcal{U})$ such that $\sigma \theta_1$ agrees with the diagonal embedding, we get the desired approximation in $A^\mathcal{U}$, and thereby, also in $A$.
\end{proof}

\end{document}